\documentclass[10pt]{article}
\usepackage{amssymb}
\usepackage{amsmath}
\usepackage{amsfonts}
\usepackage{amsthm}
\usepackage{graphics,graphicx}
\usepackage{float}
\usepackage{mathtools}
\usepackage[dvipsnames]{xcolor}
\usepackage{url}

\usepackage[english]{babel}
\usepackage{authblk}

\newtheorem{theorem}{Theorem}
\newtheorem{lemma}{Lemma}

\newtheorem{corollary}{Corollary}

\newtheorem{remark}{Remark}

\newtheorem{o-problem}{Open problem}

\title{An upper bound on the number of relevant variables for Boolean functions on the Hamming graph}
\author[1]{Alexandr Valyuzhenich\thanks{Postdoctoral Research Station of Mathematics, Hebei Normal University, Shijiazhuang 050024, PR China; Email address: graphkiper@mail.ru}}

\date{}
\begin{document}
\maketitle

\begin{abstract}
The spectrum of a complex-valued function $f$ on $\mathbb{Z}_{q}^n$ is the set
$\{|u|:u\in \mathbb{Z}_q^n~\mathrm{and}~\widehat{f}(u)\neq 0\}$, where
$|u|$ is the Hamming weight of $u$ and $\widehat{f}$ is the Fourier transform of $f$.
Let $1\leq d'\leq d\leq n$. 
In this work, we study Boolean functions on $\mathbb{Z}_{q}^n$, $q\geq 3$, whose spectrum is a subset of $\{0\}\cup \{d',\ldots,d\}$.
We prove that such functions have at most $\frac{d}{2}\cdot \frac{q^{d+d'}}{2^{d'}(q-1)^{d'}}$ relevant variables for $d'+d\leq n+1$.
In particular, we prove that any Boolean function of degree $d$ on $\mathbb{Z}_{q}^n$, $q\geq 3$, has at most $\frac{dq^{d+1}}{4(q-1)}$ relevant variables.
We also show that any equitable 2-partition of the Hamming graph $H(n,q)$, $q\geq 3$, associated with the eigenvalue $n(q-1)-qd$ has at most 
$\frac{d}{2}\cdot \frac{q^{2d}}{2^d(q-1)^{d}}$ relevant variables for $d\leq \frac{n+1}{2}$.
\end{abstract}

\section{Introduction}\label{Sec:Intro}
The {\em Hamming graph} $H(n,q)$ is defined as follows.
The vertex set of $H(n,q)$ is $\mathbb{Z}_{q}^n$, and two vertices are adjacent if they differ in exactly one position.
The {\em spectrum} of a function $f:\mathbb{Z}_{q}^n\rightarrow \mathbb{C}$ is the set 
$\{|u|:u\in \mathbb{Z}_q^n~\mathrm{and}~\widehat{f}(u)\neq 0\}$, where
$|u|$ is the Hamming weight of $u$ and $\widehat{f}$ is the Fourier transform of $f$.
The {\em degree} of a function $f:\mathbb{Z}_{q}^n\rightarrow \mathbb{C}$ is the largest element of the spectrum of $f$.
Given a function $f$ on $\mathbb{Z}_{q}^n$, a variable $x_i$, $1\leq i\leq n$, is called {\em relevant} (or {\em essential}) if there exist
$a_1,\ldots,a_{i-1},a_{i+1},\ldots,a_n\in \mathbb{Z}_q$ and $b_1,b_2\in \mathbb{Z}_q$ such that
$$f(a_1,\ldots,a_{i-1},b_1,a_{i+1},\ldots,a_n)\neq f(a_1,\ldots,a_{i-1},b_2,a_{i+1},\ldots,a_n).$$

A classical theorem of Nisan and Szegedy \cite{NS94} states that any Boolean function of degree $d$ on $\mathbb{Z}_{2}^n$ has at most $d\cdot 2^{d-1}$ relevant variables.
In 2020, Chiarelli, Hatami and Saks \cite{CHS20} improved the bound to $6.614\cdot 2^{d}$,
and then Wellens \cite{W22} improved it further to $4.394\cdot 2^{d}$.
The same problem for Boolean functions of degree $d$ on the Johnson graph was studied in \cite{FI19J,F23}.

In this work, we consider Boolean functions of a bounded degree on $\mathbb{Z}_{q}^n$ 
(or, equivalently, on the Hamming graph $H(n,q)$) for $q\geq 3$.
Any Boolean function of degree $1$ on $\mathbb{Z}_{q}^n$ has at most one relevant variable.
This fact is a very special case of the result obtained by Meyerowitz in \cite{Mey03}.
Another proof of this fact can be found in \cite{FI19}.
Boolean functions of degree $2$ on $\mathbb{Z}_{q}^n$  whose spectrum is $\{0,2\}$ were studied in \cite{MV20,MTV24}.
In particular, Mogilnykh and Valyuzhenich \cite{MV20} showed that such functions have at most $\max (4,\frac{q}{2}+1)$ relevant variables.
The remark of Filmus and Ihringer at the end of \cite{FI19J} and Wellens's result imply that any Boolean function of degree $d$ on $\mathbb{Z}_{q}^n$ has at most 
$4.394\cdot 2^{\lceil \log_2 q\rceil d}$ relevant variables.

Let $1\leq d'\leq d\leq n$. 
In this work, we study Boolean functions on $\mathbb{Z}_{q}^n$, $q\geq 3$, whose spectrum is a subset of $\{0\}\cup \{d',\ldots,d\}$.
We prove that such functions have at most $\frac{d}{2}\cdot \frac{q^{d+d'}}{2^{d'}(q-1)^{d'}}$ relevant variables for $d'+d\leq n+1$.
In particular, we prove that any Boolean function of degree $d$ on $\mathbb{Z}_{q}^n$, $q\geq 3$, has at most $\frac{dq^{d+1}}{4(q-1)}$ relevant variables.
This bound improves the bound $4.394\cdot 2^{\lceil \log_2 q\rceil d}$ in the following cases:

\begin{itemize}
  \item $q\in \{3,5,6,7\}$;
  \item $2^{k-1}<q\leq 15\cdot 2^{k-4}$, where $k\geq 4$.
\end{itemize}

Boolean functions on $\mathbb{Z}_{q}^n$ whose spectrum is $\{0,d\}$ are in one-to-one correspondence with
equitable $2$-partitions of $H(n,q)$ associated with the eigenvalue $n(q-1)-qd$.
Therefore, our result for $d'=d$ implies that any equitable 2-partition of $H(n,q)$, $q\geq 3$, associated with the eigenvalue $n(q-1)-qd$ has at most 
$\frac{d}{2}\cdot \frac{q^{2d}}{2^d(q-1)^{d}}$ relevant variables for $d\leq \frac{n+1}{2}$.

The paper is organized as follows. In Section \ref{Sec:Def}, we introduce basic definitions.
In Section \ref{Sec:Prelim}, we give preliminary results.
In Section \ref{Sec:Main}, we prove the main results.

\section{Basic definitions}\label{Sec:Def}

For a positive integer $n$, denote $[n]=\{1,\ldots,n\}$.
The weight of a vector $u\in \mathbb{Z}_{q}^n$, denoted by $|u|$, is the number of its non-zero coordinates.
The support of a function $f$ is denoted by $S(f)$.

The {\em Hamming graph} $H(n,q)$ is defined as follows.
The vertex set of $H(n,q)$ is $\mathbb{Z}_{q}^n$, and two vertices are adjacent if they differ in exactly one coordinate.
This graph has $n+1$ distinct eigenvalues $\lambda_k(n,q)=n(q-1)-q\cdot k$, where $0\leq k\leq n$.
Denote by $E_{n,q}$ and $A_{n,q}$ the edge set and the adjacency operator of $H(n,q)$.

The set of all complex-valued functions on $\mathbb{Z}_{q}^n$, denoted by $U(n,q)$, forms a vector space over $\mathbb{C}$.
Define an inner product on this vector space as follows:
$$\langle f,g \rangle=\frac{1}{q^n}\sum_{x\in \mathbb{Z}_{q}^n}f(x)\overline{g(x)}.$$

For every $0\leq k\leq n$, we denote $$U_k(n,q)=\{f:\mathbb{Z}_{q}^n\rightarrow \mathbb{C}: A_{n,q}f=\lambda_k(n,q)f\}.$$
In fact, $U_k(n,q)$ is a complex eigenspace of $A_{n,q}$ associated with $\lambda_k(n,q)$.
The direct sum of subspaces
$$U_k(n,q)\oplus U_{k+1}(n,q)\oplus\ldots\oplus U_m(n,q),$$ where $0\leq k\leq m\leq n$, is denoted by $U_{[k,m]}(n,q)$.

For every $u\in \mathbb{Z}_{q}^n$, define a {\em character} $\chi_u$ on $\mathbb{Z}_{q}^n$ by $\chi_{u}(x)=\omega^{\langle u,x\rangle}$, where 
$\omega=e^{2pi/q}$ is a primitive $q$th root of unity and
$\langle u,x\rangle=u_1x_1+\cdots+u_nx_n$.
In what follows, we will need the following two properties of the characters.

\begin{remark}\label{Rem:1}
The following statements hold:
\begin{enumerate}
  \item  The set $\{\chi_u:u\in \mathbb{Z}_{q}^n\}$ forms an orthonormal basis of the vector space $U(n,q)$.
  \item The equality $A_{n,q}\chi_u=\lambda_u\chi_u$, where $\lambda_u=n(q-1)-q\cdot |u|$, holds for any vector $u\in \mathbb{Z}_{q}^n$.
\end{enumerate}
\end{remark}

Every function $f\in U(n,q)$ is uniquely expressible as
$$f=\sum_{u\in \mathbb{Z}_{q}^n}\widehat{f}(u)\chi_u,$$
where $\widehat{f}(u)\in \mathbb{C}$ for all $u\in \mathbb{Z}_{q}^n$. The complex numbers $\widehat{f}(u)$, where $u\in \mathbb{Z}_{q}^n$, are called the {\em Fourier coefficients} of $f$.
The {\em spectrum} of a function $f\in U(n,q)$ is the set 
$\{|u|:u\in \mathbb{Z}_q^n~\mathrm{and}~\widehat{f}(u)\neq 0\}$.
The {\em degree} of a function $f\in U(n,q)$ is $\max\{|u|:u\in \mathbb{Z}_q^n~\mathrm{and}~\widehat{f}(u)\neq 0\}$.

Let $f\in U(n,q)$, let $i\in [n]$ and let $a,b\in \mathbb{Z}_q$, where $a\neq b$.
We define a function $f_{i,a,b}$ on $\mathbb{Z}_{q}^{n-1}$ as follows:
for any vector $y=(y_1,\ldots,y_{i-1},y_{i+1},\ldots,y_n)$ from $\mathbb{Z}_{q}^{n-1}$
$$f_{i,a,b}(y)=f(y_1,\ldots,y_{i-1},a,y_{i+1},\ldots,y_n)-f(y_1,\ldots,y_{i-1},b,y_{i+1},\ldots,y_n).$$

Let $f\in U(n,q)$, let $i\in [n]$ and let $a,b\in \mathbb{Z}_q$, where $a\neq b$.
Denote by $\nu(f)$ the number of edges $\{x,y\}\in E_{n,q}$ such that $f(x)\neq f(y)$; and 
denote by $\nu_{i,a,b}(f)$ the number of edges $\{x,y\}\in E_{n,q}$ such that $x_i=a$, $y_i=b$ and $f(x)\neq f(y)$.

Note that 
$$\nu(f)=\sum_{\substack{i\in [n]\\ a,b\in \mathbb{Z}_q:a<b}}\nu_{i,a,b}(f).$$
On the other hand, we have $\nu_{i,a,b}(f)=|S(f_{i,a,b})|$.
Therefore,

\begin{equation}\label{Eq:1}
\nu(f)=\sum_{\substack{i\in [n]\\ a,b\in \mathbb{Z}_q:a<b}}|S(f_{i,a,b})|.
\end{equation}

Given a function $f$ on $\mathbb{Z}_{q}^n$, a variable $x_i$, $i\in [n]$, is called {\em relevant} if there exist
$a_1,\ldots,a_{i-1},a_{i+1},\ldots,a_n\in \mathbb{Z}_q$ and $b_1,b_2\in \mathbb{Z}_q$ such that
$$f(a_1,\ldots,a_{i-1},b_1,a_{i+1},\ldots,a_n)\neq f(a_1,\ldots,a_{i-1},b_2,a_{i+1},\ldots,a_n).$$

Given a function $f$ on $\mathbb{Z}_{q}^n$, an index $i\in [n]$ is called {\em relevant} if the corresponding variable $x_i$ is relevant.

\section{Preliminaries}\label{Sec:Prelim}
The following two formulas can be proved by standard methods (the proof for $q=2$ can be found in \cite{NS94}). We give the proof for completeness.
\begin{lemma}\label{L:X(f)}
For any function $f:\mathbb{Z}_q^n\rightarrow \{-1,1\}$, the following equalities hold:
\begin{gather} 
\sum_{u\in \mathbb{Z}_q^n}|\widehat{f}(u)|^2=1, \label{Eq:2} \\ 
\nu(f)=\frac{q^{n+1}}{4}\sum_{u\in \mathbb{Z}_q^n}|u|\cdot |\widehat{f}(u)|^2. \label{Eq:3}
\end{gather}

\end{lemma}
\begin{proof}

By Remark \ref{Rem:1}, we have

$$\langle f,f \rangle=\langle \sum_{u\in \mathbb{Z}_q^n}\widehat{f}(u)\chi_u, \sum_{u\in \mathbb{Z}_q^n}\widehat{f}(u)\chi_u \rangle=\sum_{u\in \mathbb{Z}_q^n}|\widehat{f}(u)|^2.$$

On the other hand, $$\langle f,f \rangle=\frac{1}{q^n}\sum_{x\in \mathbb{Z}_q^n}|f(x)|^2=1.$$
This immediately implies the equality (\ref{Eq:2}).
Using Remark \ref{Rem:1}, we obtain

\begin{equation}\label{Eq:4}
\begin{split}
& \langle A_{n,q}f,f \rangle=\langle \sum_{u\in \mathbb{Z}_q^n}\lambda_u\widehat{f}(u)\chi_u, \sum_{u\in \mathbb{Z}_q^n}\widehat{f}(u)\chi_u \rangle=
\sum_{u\in \mathbb{Z}_q^n}\lambda_u\cdot|\widehat{f}(u)|^2= \\
& =n(q-1)-q\sum_{u\in \mathbb{Z}_q^n}|u|\cdot |\widehat{f}(u)|^2.
\end{split}
\end{equation}

On the other hand, 
\begin{equation}\label{Eq:5}
\langle A_{n,q}f,f \rangle=\frac{2}{q^n}\sum_{\{x,y\}\in E_{n,q}}f(x)f(y)=n(q-1)-\frac{4\nu(f)}{q^n}.
\end{equation}

Subtracting (\ref{Eq:5}) from (\ref{Eq:4}), we get the equality (\ref{Eq:3}).
\end{proof}

Using Lemma \ref{L:X(f)} for Boolean functions of degree at most $d$, we immediately obtain the following result.

\begin{corollary}\label{Cor:X(f)}
Let $f$ be a Boolean function of degree at most $d$ on $\mathbb{Z}_{q}^n$. Then $\nu(f)\leq \frac{d}{4}\cdot q^{n+1}$.
\end{corollary}

The following lemma shows the connection between the eigenspaces of the Hamming graphs $H(n,q)$ and $H(n-1,q)$.
\begin{lemma}\label{L:Reduction}
Let $f\in{U_{[k,m]}(n,q)}$, let $i\in [n]$, and let $a,b\in \mathbb{Z}_q$, $a\neq b$. Then $f_{i,a,b}\in U_{[k-1,m-1]}(n-1,q)$.
\end{lemma}

The proof of Lemma \ref{L:Reduction} can be found, for example, in \cite[Lemma 4]{VV19}.
The following lower bound on the support size of functions on $\mathbb{Z}_{q}^n$ was proved in \cite{VV19}.
\begin{theorem}\label{Th:Bound}

Let $f\in{U_{[k,m]}(n,q)}$, where $q\geq 3$, $k+m\le n$ and $f\not\equiv 0$. Then $$|S(f)|\geq 2^{k}\cdot (q-1)^{k}\cdot q^{n-k-m}.$$
\end{theorem}
We note that the bound from Theorem \ref{Th:Bound} is sharp.

\section{Main results}\label{Sec:Main}
The main result of this paper is the following.
\begin{theorem}\label{Th:Main}
Let $f\in U_0(n,q)\oplus U_{[d',d]}(n,q)$, where $q\geq 3$, $1\leq d'\leq d\leq n$ and $d'+d\leq n+1$.
If $f$ is Boolean, then $f$ has at most $\frac{d}{2}\cdot \frac{q^{d+d'}}{2^{d'}(q-1)^{d'}}$ relevant variables.
\end{theorem}
\begin{proof}
Suppose $f$ has $n'$ relevant variables.
Let us consider a relevant index $i\in [n]$. Define a partition $A_1\cup\ldots \cup A_t$ of $\mathbb{Z}_q$ as follows:
$a,b\in A_r$ for some $r\in [t]$ if and only if $f_{i,a,b}\equiv 0$.

Since $i$ is relevant, we have $t\geq 2$.
For every $r\in [t]$, denote $q_r=|A_r|$. Note that

\begin{equation}\label{Eq:6}
|\{(a,b):a,b\in \mathbb{Z}_q,a<b,f_{i,a,b}\not\equiv 0\}|=\sum_{1\leq r<s\leq t}q_rq_s\geq q_1(q_2+\cdots+q_t)\geq q-1 . 
\end{equation}

By Lemma \ref{L:Reduction}, we have $f_{i,a,b}\in U_{[d'-1,d-1]}(n-1,q)$ for all $a,b\in \mathbb{Z}_q$, where $a\neq b$.
Applying Theorem \ref{Th:Bound} for $f_{i,a,b}$, where $a\neq b$, we get
\begin{equation}\label{Eq:7}
|S(f_{i,a,b})|\geq 2^{d'-1}\cdot (q-1)^{d'-1}\cdot q^{n-d'-d+1}
\end{equation}
or  $f_{i,a,b}\equiv 0$.
Combining (\ref{Eq:6}) and (\ref{Eq:7}), we obtain that the inequality
\begin{equation}\label{Eq:8}
\sum_{a,b\in \mathbb{Z}_q:a<b}|S(f_{i,a,b})|\geq 2^{d'-1}\cdot (q-1)^{d'}\cdot q^{n-d'-d+1}
\end{equation}
holds for any relevant index $i\in [n]$.

Using the equality (\ref{Eq:1}) and the inequality (\ref{Eq:8}), we obtain the following lower bound on $\nu(f)$:

$$\nu(f)\geq n'\cdot 2^{d'-1}\cdot (q-1)^{d'}\cdot q^{n-d'-d+1}.$$

On the other hand, we have $\nu(f)\leq \frac{d}{4}\cdot q^{n+1}$ due to Corollary \ref{Cor:X(f)}. Therefore,

$$n'\leq \frac{d}{2}\cdot \frac{q^{d+d'}}{2^{d'}(q-1)^{d'}}.$$

\end{proof}

Applying Theorem \ref{Th:Main} for $d'=1$ and $d'=d$ respectively, we obtain the following results.

\begin{corollary}\label{Cor:2}
Any Boolean function of degree $d$ on $\mathbb{Z}_{q}^n$, $q\geq 3$, has at most $\frac{dq^{d+1}}{4(q-1)}$ relevant variables.
\end{corollary}

\begin{corollary}\label{Cor:3}
Let $f\in U_0(n,q)\oplus U_d(n,q)$, where $q\geq 3$ and $1\leq d\leq \frac{n+1}{2}$.
If $f$ is Boolean, then $f$ has at most $\frac{d}{2}\cdot \frac{q^{2d}}{2^d(q-1)^{d}}$ relevant variables.
\end{corollary}

Corollary \ref{Cor:3} can be reformulated in terms of equitable $2$-partitions of $H(n,q)$.
Firstly, we recall some definitions.
Let $G$ be a graph. An $r$-partition $C_1\cup\ldots \cup C_r$ of the vertex set of $G$ is called {\em equitable} if
for any $i,j\in [r]$ there is a constant $s_{i,j}$ such that any vertex of $C_i$ has exactly $s_{i,j}$ neighbors in $C_j$.
The matrix $S=(s_{i,j})_{1\leq i,j\leq r}$ is the {\em quotient matrix} of the partition.
One of the important properties of equitable partitions is the following: 
the eigenvalues of the quotient matrix of an equitable partition of a graph $G$ are also eigenvalues of the adjacency matrix of $G$
(see, for example, \cite[Lemma 2.2]{G93}).
There are many different names in the literature for equitable partitions of graphs: perfect coloring, partition design, regular partition.
We also note that every class of an equitable 2-partition of a graph is a completely regular code of radius one.
An interesting discussion of equitable partitions and related concepts can be found in \cite{I:Blog}.

Fon-Der-Flaass initiated the systematic study of equitable 2-partitions of the Hamming graph in \cite{FDF07SMJ,FDF07Bound}
(he considered the case $q=2$).
In particular, he proposed several iterative constructions of equitable 2-partitions of $H(n,2)$.
He also proved several important necessary conditions for the existence of equitable 2-partitions of $H(n,2)$ with a given quotient matrix.
For more information on equitable 2-partitions of the Hamming graphs we refer the reader to a recent survey \cite{BKMTV21}.
Suppose $C_1\cup C_2$  is an equitable $2$-partition of $H(n,q)$ with the quotient matrix
$$S=\begin{pmatrix}
a & b\\
c & d\\
\end{pmatrix}.$$
The matrix $S$ has two different eigenvalues: the {\em trivial} eigenvalue $\theta_0=n(q-1)$ (the degree of $H(n,q)$) and the {\em second} eigenvalue $\theta_1=a-c$.
Moreover, we have $\theta_1=\lambda_k(n,q)$ for some $0\leq k\leq n$.
We say that an equitable $2$-partition of $H(n,q)$ has {\em degree} $d$ if the second eigenvalue of its quotient matrix is $\lambda_d(n,q)$.

Suppose that $C_1\cup C_2$ is an equitable $2$-partition of $H(n,q)$ of degree $d$.
We have $\mathbf{1}_{C_1}\in U_{0}(n,q)\oplus U_d(n,q)$, where $\mathbf{1}_{C_1}$ is the characteristic function of $C_1$ in $\mathbb{Z}_q^n$.
Applying Corollary \ref{Cor:3} to the function $\mathbf{1}_{C_1}$, we get the following.
\begin{corollary}\label{Cor:4}
Let $q\geq 3$ and $1\leq d\leq \frac{n+1}{2}$.
Then any equitable 2-partition of $H(n,q)$ of degree $d$ has at most $\frac{d}{2}\cdot \frac{q^{2d}}{2^d(q-1)^{d}}$ relevant variables.
\end{corollary}

Finally, we compare the bounds from Theorem \ref{Th:Main}, Corollary \ref{Cor:2} and Corollary~\ref{Cor:3} with the bound $4.394\cdot 2^{\lceil \log_2 q\rceil d}$:

\begin{enumerate}
  
  \item If $d'>\frac{d}{8.788}\cdot \frac{q}{q-2}$, then the bound from Theorem \ref{Th:Main} is better than $4.394\cdot 2^{\lceil \log_2 q\rceil d}$ 
  for all $q\geq 3$ (we assume that $d'+d\leq n+1$).
  
  \item The bound from Corollary \ref{Cor:2} is better than $4.394\cdot 2^{\lceil \log_2 q\rceil d}$ in the following cases:

\begin{itemize}
  
  \item $q\in \{3,5,6,7\}$;
  \item $2^{k-1}<q\leq 15\cdot 2^{k-4}$, where $k\geq 4$.
\end{itemize}
  
  \item The bound from Corollary \ref{Cor:3} is better than $4.394\cdot 2^{\lceil \log_2 q\rceil d}$ for all $q\geq 3$ (we assume that $d\leq \frac{n+1}{2}$). 
\end{enumerate}


\begin{thebibliography}{99}
\bibitem{BKMTV21}
E. Bespalov, D. Krotov, A. Matiushev, A. Taranenko, K. Vorob'ev,
Perfect $2$-colorings of Hamming graphs, Journal of Combinatorial Designs 29(6) (2021) 367--396.



\bibitem{CHS20}
J. Chiarelli, P. Hatami, M. Saks, An asymptotically tight bound on the number of relevant variables in a bounded degree Boolean function,
Combinatorica 40 (2020) 237--244.

\bibitem{FI19J}
Y. Filmus, F. Ihringer, Boolean constant degree functions on the slice are juntas, Discrete Mathematics 342(12) (2019) 111614.


\bibitem{FI19}
Y. Filmus, F. Ihringer, Boolean degree 1 functions on some classical association schemes, Journal of Combinatorial Theory, Series A 162 (2019) 241--270.

\bibitem{F23}
Y. Filmus, Junta threshold for low degree Boolean functions on the slice, The Electronic Journal of Combinatorics, 30(1) (2023) \#P1.55.

\bibitem{FDF07SMJ}
D. G. Fon-Der-Flaass, Perfect 2-colorings of a hypercube, Siberian Mathematical
Journal 48(4) (2007) 740--745.


\bibitem{FDF07Bound}
 D. G. Fon-Der-Flaass, A bound on correlation immunity, Sib. Elektron. Mat. Izv. 4 (2007) 133--135.
 

 
 \bibitem{G93}
 C. D. Godsil, Algebraic Combinatorics, Chapman and Hall Mathematics Series, Chapman Hall, New York, 1993. 
 

\bibitem{I:Blog}
F. Ihringer, Translating terminology: equitable partitions and related concepts, blog entry:
\url{https://ratiobound.wordpress.com/2018/12/02/translating-terminology-equitable-partitions-and-related-concepts/}


\bibitem{Mey03}
A. Meyerowitz, Cycle-balance conditions for distance-regular graphs, Discrete Mathematics 264(1--3) (2003) 149--165.


\bibitem{MV20}
I. Mogilnykh, A. Valyuzhenich, Equitable $2$-partitions of the Hamming graphs with the second eigenvalue, Discrete Mathematics 343(11) (2020) 112039.


\bibitem{MTV24}
I. Mogilnykh, A. Taranenko, K. Vorob'ev, Completely regular codes with covering radius $1$ and the second eigenvalue in $3$-dimensional Hamming graphs,
arXiv:2403.02702, March 2024.

\bibitem{NS94}
N. Nisan, M. Szegedy, On the degree of Boolean functions as real polynomials, Computational Complexity 4 (1994) 301--313.


\bibitem{VV19}
A. Valyuzhenich, K. Vorob'ev, Minimum supports of functions on the Hamming graphs with spectral constraints,  Discrete Mathematics 342(5) (2019) 1351--1360.


\bibitem{W22}
J. Wellens, Relationships between the number of inputs and other complexity measures of Boolean functions,
arXiv:2005.00566v2, 2022. 


\end{thebibliography}
\end{document}